\documentclass[letterpaper, 10 pt, conference]{ieeeconf}  % Comment this line out
\usepackage{color}
\usepackage{amsmath}
\usepackage{amssymb}
\usepackage{graphicx}
\usepackage{subfigure} %Amirali added this.
\usepackage{comment,xspace}
\usepackage{algorithm,algorithmic}
\usepackage{fancybox}
\newtheorem{theorem}{Theorem}[section]

\newtheorem{lemma}[theorem]{Lemma}
\newtheorem{remark}{Remark}[section]

\newtheorem{corollary}{Corollary}[section]

\IEEEoverridecommandlockouts                              % This command is only
\usepackage{cite}
\usepackage{graphics} % for pdf, bitmapped graphics files
\usepackage{amsmath} % assumes amsmath package installed
\usepackage{amssymb}  % assumes amsmath package installed
\usepackage{subfigure}
\usepackage{pifont}

\usepackage{scalefnt}
\usepackage{float}
\usepackage{caption}
\newtheorem{mydef}{Definition}

\newcommand{\grad}{\nabla}

\newcommand{\mymargin}[1]{\marginpar{\color{red}\tiny\ttfamily#1}}
\renewcommand{\mymargin}[1]{}

\title{\LARGE \bf Complexity of Ten Decision Problems in\\ Continuous Time Dynamical Systems}

\author{Amir Ali Ahmadi, Anirudha Majumdar, and Russ Tedrake
\thanks{Amir Ali Ahmadi is with the Krannert School of Management
at Purdue University and currently a Goldstine Fellow at the
Department of Business Analytics and Mathematical Sciences of the
IBM Watson Research Center; Email: \texttt{a\_a\_a@mit.edu}.
Anirudha Majumdar and Russ Tedrake are with the Computer Science and Artificial Intelligence Laboratory (CSAIL) at the Massachusetts Institute of Technology. Email: \texttt \{anirudha,russt\}@mit.edu.} }

\begin{document}

\maketitle
\thispagestyle{empty}
\pagestyle{empty}

%%%%%%%%%%%%%%%%%%%%%%%%%%%%%%%%%%%%%%%%%%%%%%%%%%%%%%%%%%%%%%%%%%%%%%%%%%%%%%%%
\begin{abstract}
We show that for continuous time dynamical systems described
by polynomial differential equations of
modest degree (typically equal to three), the following decision
problems which arise in numerous areas of systems and control
theory cannot have a polynomial time (or even pseudo-polynomial
time) algorithm unless P=NP: local attractivity of an equilibrium point, stability of an equilibrium point in the sense of Lyapunov, boundedness of trajectories, convergence of all trajectories in a ball to a given equilibrium point, existence of a quadratic Lyapunov function, invariance of a ball, invariance of a quartic
semialgebraic set under linear dynamics, local collision avoidance, and existence of a stabilizing control law. We also extend our earlier NP-hardness proof of testing local asymptotic stability for polynomial vector fields to the case of trigonometric differential equations of degree four.

\end{abstract}

%%%%%%%%%%%%%%%%%%%%%%%%%%%%%%%%%%%%%%%%%%%%%%%%%%%%%%%%%%%%%%%%%%%%%%%%%%%%%%%%

\section{Introduction}

Polynomial and trigonometric differential equations appear ubiquitously in a variety of application domains including robotics, economics, mathematical biology, and chemical engineering, among others~\cite{Braun93}. The equations of motion for most robotic systems for example can be described by the familiar \emph{manipulator equations} \cite{Murray94} which give rise to systems of differential equations that are a mixture of polynomial and trigonometric terms in the state variables. In mathematical biology and economics, polynomial differential equations such as the Lotka-Volterra model and its variants \cite{Freedman80} are used to model population dynamics and competition among entities in an economy. The dynamics of many chemical processes \cite{Bequette98} are also naturally modeled by polynomial differential equations. Aside from these specific examples, differential equations in numerous application domains are commonly \emph{approximated} as polynomials.

While mature computational tools exist for the numerical solution of such differential equations \cite{Atkinson09}, in most of the application domains described above, one is \emph{not} typically interested in \emph{particular} solutions of the system. Rather, \emph{qualitative properties} of the differential equations are of central importance. For example, one may be interested in the safety of a robot performing a certain dynamic task, or in determining if the population of certain species diminishes below a critical threshold. The former example is related to the \emph{stability} of the control system employed by the robot while the latter can be addressed by defining an ``acceptable'' set of population numbers and asking whether this set is \emph{invariant} (i.e. if the populations start off in this set, will they always remain within the set?). In a similar vein, one can ask if trajectories in a model of epidemic spread remain \emph{bounded}. 

The study of such qualitative properties of differential equations has an extensive literature and numerous algorithms have been proposed for addressing these questions computationally. However, for all but the simplest cases (e.g. the case of linear systems), the problems still lack satisfactory (i.e. exact and efficient) algorithms. This observation motivates the study of the fundamental computational complexity of these problems in order to establish theoretical bounds on the efficiency of algorithms that attempt to answer these questions. Such complexity results may play an important role in shaping the search for practical algorithms for these problems by limiting the kinds of algorithms one can possibly hope to obtain. Further, by understanding exactly where the complexity of these problems stems from, we can hope to find \emph{approximations} and \emph{relaxations} that are more amenable to efficient solutions while still maintaining practical relevance.

Questions of complexity related to qualitative properties of differential equations have long been of theoretical interest. A natural question one can ask is if the stability of a system of polynomial differential equations can be decided by a Turing machine in finite time. In \cite{Arnold_Problems_for_Math}, Arnold made a well-known conjecture that the contrary is true; i.e. the question is undecidable. To the authors' knowledge, even though some variants of the question have been studied and answered \cite{Undecidability_vec_fields_survey,Costa_Doria_undecidabiliy}, the question in its original form is so far unresolved. Although the results in this paper do not resolve Arnold's question, they provide lower bounds on the computational complexity of deciding local asymptotic stability and several similar and related problems.

The primary challenge in establishing such complexity results lies in relating the properties of the \emph{continuous} solutions of polynomial and trigonometric differential equations to the \emph{combinatorial} problems for which complexity results have been previously established. Explicitly mapping trajectories of a system (which typically one does not have access to exactly) to objects in combinatorial problems seems to be a hopeless approach. The main idea that allows us to by-pass this apparent challenge is to relate the combinatorial problem to properties of \emph{Lyapunov functions} that prove stability/invariance of differential equations. All the results in this paper exploit this idea in one way or another. 

%Another important class of related questions that follow as corollaries of the main results in this paper relate to the %computational complexity of \emph{designing controllers} that achieve stability/invariance for control systems described by %polynomial and trigonometric differential equations. The particular decision problem we study is of the form: For a given %control system, does there exist a controller that makes the system stable (either locally/globally) or a given set invariant?

The organization of the paper is as follows. After stating some preliminaries in Section~\ref{sec:prelim}, we show in Section \ref{sec:trig} that deciding local asymptotic stability for trigonometric polynomials of degree four is strongly NP-hard. While this result is an extension of the results presented in \cite{AAA_ACC12_Cubic_Difficult} (which proves the corresponding result for cubic polynomial vector fields), the decision problem is of independent interest, particularly in the field of robotics. This is due to the fact that most mechanical systems can be modeled by the \emph{manipulator equations}, which result in vector fields whose degrees are dominated by trigonometric terms. In Section \ref{sec:poly}, we prove that the following decision problems are strongly NP-hard for polynomial vector fields of degree $d$:

\begin{itemize}
 \item Invariance of a ball ($d=3$),
 \item Invariance of a basic semialgebraic set defined by a quartic polynomial ($d=1$),
\item Inclusion of the unit ball in the region of attraction of an equilibrium point ($d=3$),
\item Local attractivity of an equilibrium point ($d=3$),
 \item Stability of an equilibrium point in the sense of Lyapunov ($d=4$), 
 \item Boundedness of trajectories ($d=3$),
 \item Existence of a quadratic Lyapunov function ($d=3$),
 \item Local collision avoidance ($d=4$),
 \item Existence of a stabilizing controller ($d=3$).

\end{itemize}
These notions are all formally defined in Section \ref{sec:poly}. The input to these problems is an ordered list of coefficients (expressed as rational numbers) defining the polynomial or trigonometric vector field. Establishing NP-hardness of these problems implies that unless P=NP, it is not possible to provide an algorithm that can have a running time bounded by a polynomial in the number of bits required to represent the input. Further, all the NP-hardness results in this paper are in the \emph{strong} sense (as opposed to weakly NP-hard problems like KNAPSACK or SUBSET SUM). This implies that the problems remain NP-hard even when the bit length of the coefficients (i.e. the input) is $O(log(n))$ (here, $n$ is the dimension of the state space). Unless P=NP, even pseudo-polynomial time algorithms cannot exist for strongly NP-hard problems; see  \cite{GareyJohnson_Book} for more details and definitions. In particular, our results suggest that none of the numerous recent techniques for systems analysis based on convex optimization (e.g. in terms of linear programs, linear matrix inequalities, or sum of squares programs) can be exact, unless the size of the formulated optimization problems are exponential in the input.

%rule out the possibility of exact algorithms that are based on convex optimization (e.g. based on Linear Matrix Inequality %approaches). In recent years, approaches that search for \emph{sums-of-squares} Lyapunov functions using semidefinite %programming \cite{PhD:Parrilo}  have been proposed for deciding stability of polynomial systems. Our results suggest the %non-existence of sums-of-squares Lyapunov functions of ``small'' degree.

We refer the reader interested in computational complexity in systems and control to the outstanding survey papers~\cite{Survey_CT_Computation},
\cite{Undecidability_vec_fields_survey},
\cite{Sontag_complexity_comparison},
\cite{BlTi_complexity_3classes}, \cite{BlTi1} and references therein.

\section{A few preliminaries on forms}\label{sec:prelim}

Many of the results in this paper will make use of \emph{homogeneous} polynomials. A multivariate polynomial $p:\mathbb{R}^n\rightarrow\mathbb{R}$ is homogeneous
(of degree $d$) if it satisfies $p(\lambda x)=\lambda^d p(x)$ for
all $x\in\mathbb{R}^n$ and all $\lambda\in\mathbb{R}$. This condition is equivalent to all monomials of $p$ having the same degree. A homogeneous polynomial is also called a \emph{form}. Observe that products of forms are again forms, and that the components of the gradient of a form are forms of one fewer degree. We will make frequent references to the following useful identity for homogeneous functions due to Euler:
$$p(x) = \frac{1}{d}x^T\grad p(x).$$
Here, $p$ is a homogeneous function of degree $d$ and $\nabla p$ denotes its gradient vector. The identity is easily derived by differentiating both sides of the above equation with respect to $\lambda$ and setting $\lambda = 1$.

The degree of a polynomial vector field $\dot{x} = f(x)$, with $f:\mathbb{R}^n\rightarrow\mathbb{R}^n$, is
defined to be the largest degree of the components of $f$. We say that the vector field $f$ is homogeneous if all components of $f$ are forms of the same degree. Finally, a form $p:\mathbb{R}^n\rightarrow\mathbb{R}$ is said to be \emph{positive definite} if $p(x)>0$ for all nonzero $x$ in $\mathbb{R}^n$.

%A large body of work in nonlinear control has been focused on homogeneous vector fields; see
%e.g.~\cite{Stability_homog_poly_ODE}, \cite{Stabilize_Homog},
%\cite{homog.feedback}, \cite{Baillieul_Homog_geometry},
%\cite{Cubic_Homog_Planar}, \cite{HomogHomog},
%\cite{homog.systems}. Since our results are negative in nature,
%their validity for homogeneous polynomial systems obviously also
%implies their validity for all polynomial systems.

\section{Complexity of deciding local asymptotic stability of trigonometric vector fields}
\label{sec:trig}

In this section, we prove that deciding local asymptotic stability of trigonometric vector fields of degree four is strongly NP-hard.

\begin{mydef}
The zero equilibrium point of a dynamical system $\dot{x} = f(x)$ is \emph{stable in the sense of Lyapunov} if $\forall \ \epsilon > 0$, $\exists \ \delta > 0$ such that $\|x(0)\| < \delta \implies \|x(t)\| < \epsilon, \forall t > 0.$ We say that the equlibrium point is \emph{locally asymptotically stable} if it is stable in the sense of Lyapunov and there exists $\epsilon > 0$ such that $\|x(0)\| < \epsilon \implies \lim_{t \to \infty} x(t) = 0$.
\end{mydef}

%\begin{mydef}
%A continuous time dynamical system $\dot{x} = f(x)$ is \emph{locally asymptotically stable} if it is stable in the sense of %Lyapunov and there exists $\epsilon > 0$ such that $\|x(0)\| < \epsilon \implies \lim_{t \to \infty} x(t) = 0$.
%\end{mydef}

\begin{theorem}\label{thm:las.trig}
Given a trigonometric vector field of degree four, it is strongly
NP-hard to decide if it is locally asymptotically stable.
\end{theorem}

%The main idea of the proof of this theorem will be to give a polynomial time reduction from our base combinatorial problem %(ONE-IN-THREE 3SAT) to checking properties of a Lyapunov function that proves local asymptotic stability of the vector field.

The proof will be via a reduction from ONE-IN-THREE 3SAT, which is known to be NP-complete \cite{GareyJohnson_Book}. An instance of ONE-IN-THREE 3SAT consists of an expression made of conjunctions of clauses. Each clause is the logical OR of three literals, and each literal is either a variable or its negation. The problem is to decide if there is a boolean assignment\footnote{An assignment to a variable will denote an element from $\{0,1\}$ with $0$ corresponding to false and $1$ corresponding to true.} of the variables that results in the expression being true and each clause having \emph{exactly one} true literal.

Given an instance of ONE-IN-THREE 3SAT in the variables $b_1, \ldots, b_n$, we first construct a degree 4 trigonometric polynomial. To avoid introducing unnecessary notation, we present this construction on a single instance of ONE-IN-THREE 3SAT. This example should elucidate how the procedure works in the general case. Given an instance of ONE-IN-THREE 3SAT such as:
$$(b_1 \lor b_2 \lor b_3) \land (b_1 \lor \bar{b}_4 \lor b_5)$$
we construct the following quartic trigonometric polynomial:
\begin{flalign}
  t(x) = & \ \Sigma_{i=1}^{n} \sin^2(x_i)(1 - \sin(x_i))^2 \nonumber \\
       & + [\sin(x_1) + \sin(x_2) + \sin(x_3) - 1]^2  \nonumber \\
       & + [\sin(x_1) + (1 - \sin(x_4)) + \sin(x_5) - 1]^2 \nonumber
\end{flalign}

The first term in $t(x)$ is not specific to the particular instance under consideration and will always appear in our construction. The following terms are constructed by taking each clause appearing in the boolean formula and replacing a variable $b_i$ with $\sin(x_i)$ and substituting $+$ in place of $\lor$. If the negation of a variable appears in the clause, we replace it by $1 - \sin(x_i)$. Each resulting expression (corresponding to an individual clause) is subtracted by $1$ and then squared.
%  We note that $t(x) > 0, \ \forall x \in \RR^n$ if and only if the ONE-IN-THREE 3SAT instance is not satisfiable. To see this, first note that by construction, $t(x)$ is a sum of squares and hence nonnegative. Also, by construction, the only zeros occur when $sin(x_i) \in \{0,1\}$. Suppose there is an assignment of variables, $x^\star$, that satisfies the ONE-IN-THREE 3SAT instance. Then, $t(x)$ has a zero if we substitute $x_i = sin^{-1}(x_i^\star)$. Conversely, if the ONE-IN-THREE 3SAT instance is unsatisfiable, at least one of the terms in $t(x)$ must be positive by construction.

We then introduce a single new variable $y$ and define $z := [x, y]^T$. We construct a new trigonometric polynomial:
\begin{flalign}
t_h(z) = & \ \Sigma_{i=1}^{n} \sin^2(x_i)(\sin(y) - \sin(x_i))^2 \nonumber \\
       & + [\sin(x_1)\sin(y) + \sin(x_2)\sin(y) \dots \nonumber \\
       & + \sin(x_3)\sin(y) - \sin^2(y)]^2  \nonumber \\
       & + [\sin(x_1)\sin(y) + (\sin^2(y) - \sin(x_4)\sin(y)) \dots \nonumber \\
       & + \sin(x_5)\sin(y) - \sin^2(y)]^2 \nonumber
\end{flalign}
Here, we first replaced $\sin(x_i)$ with $\frac{\sin(x_i)}{\sin(y)}$ and then multiplied the entire resulting expression by $\sin^4(y)$. Observe that by doing so, $t_h$ becomes a homogeneous function of $\sin(z)$.

%Denoting $s = [\sin(x_1), \dots, \sin(x_n), \ \sin(y)]^T$, we see that $t_h(s)$ is a quartic form in $s$. This %observation will be important in the following lemmas.

\begin{mydef}
A function $g(z)$ is \emph{locally positive definite} if there exists $\epsilon > 0$ such that $\|z\| < \epsilon, z\neq 0 \implies g(z) > 0$.
\end{mydef}

\begin{lemma}
The trigonometric polynomial $t_h(z)$ is locally positive definite if and only if the ONE-IN-THREE 3SAT instance it was derived from is unsatisfiable.
\end{lemma}

\begin{proof}
To see this, first note that by construction $t(x)$ and $t_h(z)$ are sums of squares and hence nonnegative. Also, by construction, the only zeros of $t(x)$ occur when $\sin(x_i) \in \{0,1\}$. Now, suppose that the ONE-IN-THREE 3SAT instance has a satisfying assignment of variables $b^\star$. Then, we observe that $t(x)$ must have a zero if we substitute $x_i = \sin^{-1}(b_i^\star)$. Thus, by construction, $t_h(z)$ has a zero when
\begin{equation}\nonumber
\begin{array}{rll}
x_i &=& \sin^{-1}(b_i^\star)\\
y &=& \sin^{-1}(1).
\end{array}
\end{equation}

Further, since $t_h$ is homogeneous in $\sin(z)$, for all $\alpha\geq 0$, we have $t_h(x(\alpha),y(\alpha))=0$, where
\begin{equation}\nonumber
\begin{array}{rll}
x_i(\alpha)&=& \sin^{-1}(\alpha b_i^\star)\\
y(\alpha) &=& \sin^{-1}(\alpha).
\end{array}
\end{equation}
Note that as $\alpha\rightarrow 0$, $(x(\alpha),y(\alpha))$ gets arbitrarily close to the origin. Hence $t_h(z)$ is not locally positive. 

%a quartic form in $s$, from the scaling property of forms, $t_h(z)$ must have a zero for all $z$ and $\alpha$ such that the following equations hold:
%$$\sin(x_i) = \alpha b_i^\star$$
%$$\sin(y_i) = \alpha$$
%Thus, $t_h(z)$ has a zero arbitrarily close to the origin since for $z$ arbitrarily close to the origin, there %exists $\alpha$ such that the above equations are satisfied.

Now, to prove the converse, suppose the ONE-IN-THREE 3SAT instance is unsatisfiable. We know that at least one of the clauses in the ONE-IN-THREE 3SAT instance must have either no true literals or more than one true literal for all possible assignments to the variables. Thus, at least one of the terms of $t(x)$ must always be positive.  If $t(x)$ does not have a zero, by construction, $t_h(z)$ can only have a zero when
$\sin(y) = 0$. However, substituting $\sin(y) = 0$ in $t_h(z)$ results in the expression $\Sigma_{i}^{n} \sin^4(x_i)$.
Thus, $t_h(z)$ has a zero only when $\sin(x_i) = 0$ and $\sin(y) = 0$. Hence, in some neighborhood around the origin ($-\pi < x_i < \pi$ and $-\pi < y < \pi$), these equations cannot be satisfied and thus $t_h(z)$ is locally positive.
\end{proof}

\begin{corollary}
Checking local positivity of quartic trigonometric functions is strongly NP-hard.\footnote{Our reduction from ONE-IN-THREE 3SAT to checking local positivity is clearly polynomial in length.}
\end{corollary}

\begin{lemma}
If the ONE-IN-THREE 3SAT instance is unsatisfiable, there exists a neighborhood around the origin in which $\grad t_h(z)$ does not vanish (except at the origin). 
\end{lemma}

\begin{proof}
Let $s := [\sin(x_1), \sin(x_2), \dots \sin(y)]^T$ and denote the function $t_s(s)$ to be the function $t_h(z)$ viewed as a function of $s$. Note that $t_s$ is a quartic form in $s$. The expression for the gradient of $t_h(z)$ with respect to $z$ can be written as
$$\grad t_h(z) = \mbox{diag}(\cos(z_i)) \grad_s t_s(s)$$
where $\mbox{diag}(\cos(z_i))$ is a diagonal matrix with its diagonal elements set to $\cos(z_i)$. Using Euler's identity for homogeneous functions, we have $t_s(s) = \frac{1}{4}s^T \grad_s t_s(s)$. Thus, we see that when $t_s(s)$ is nonzero, $\grad_s t_s(s)$ cannot vanish at that point. Since for a small enough neighborhood around the origin ($-\frac{\pi}{2} < x_i < \frac{\pi}{2}$ and $-\frac{\pi}{2} < y < \frac{\pi}{2}$) we know that $t_s(s)$ is positive \emph{and} $\cos(z_i)$ is nonzero for all $i$, we have that $\grad t_h(z)$ also does not vanish in this neighborhood (except at the origin).
\end{proof}

\begin{corollary} \label{cor:np_hard_local_trig}
 Given a quartic trigonometric polynomial, it is NP-hard to decide if there exists a neighborhood around the origin where the function is positive and its gradient does not vanish (except at the origin).
\end{corollary}
%\hfill$\blacksquare$

We now present a polynomial time reduction from the decision problem stated in Corollary \ref{cor:np_hard_local_trig} to the problem of checking local asymptotic stability of quartic trigonometric vector fields. This completes the proof of Theorem \ref{thm:las.trig}.

\begin{proof}[Proof of Theorem~\ref{thm:las.trig}]
Given a quartic trigonometric form $t_h(z)$, we construct the following continuous time dynamical system:
$$\dot{z} = -\grad t_h(z).$$
Observe that this a quartic trigonometric vector field. We claim that this system is locally asymptotically stable if and only if there is a neighborhood around the origin such that $t_h(z)$ is positive definite and $\grad t_h(z)$ does not vanish except at the origin. To prove the claim, we start by noting that by construction $\dot{t}_h(z)$ is always negative semidefinite:
$$\dot{t}_h(z) = \langle \grad t_h(z), -\grad t_h(z)\rangle = -\|\grad t_h(z)\|^2.$$
Suppose first that there is a neighborhood in which $t_h(z)$ is positive definite and $\grad t_h(z)$ does not vanish except at the origin. The condition on the gradient implies from the equation above that $\dot{t}_h(z)$ is locally negative definite. This, together with local positivity of $t_h$, implies that $t_h$ is a locally valid Lyapunov function for the system. Hence, local asymptotic stability follows from Lyapunov's stability theorem (see e.g.~\cite[p. 124]{Khalil:3rd.Ed}).

To see the converse, suppose that the vector field is locally asymptotically stable. There is therefore an open neighborhood $B_\delta$ around the origin, where the trajectories converge to the origin. We first observe that $\grad t_h(z)$ does not vanish in this neighborhood since if it did, $\dot{z}$ would equal zero at that point. This would contradict local asymptotic stability (since this point would be an equilibrium point and would not asymptotically converge to the origin).

Next, we prove that $t_h(z)$ is positive definite in $B_\delta$. Suppose first that there exists a point $z_0$ in $B_\delta$ such that $t_h(z_0) < 0$. Consider the trajectory that starts off at $z_0$. Since $\dot{t}_h(z) = -\|\grad t_h(z)\|^2$ is non-increasing everywhere, we see that at all points in time, the value of $t_h(z)$ remains negative and thus the trajectory cannot go to the origin (since $t_h(0) = 0$). This contradicts local asymptotic stability. To prove that $t_h(z)$ is in fact \emph{strictly} positive in $B_\delta$, suppose for the sake of contradiction that there is a point $z^\star \in B_\delta$ such that $t_h(z^\star) = 0$. Since $t_h(z)$ was just shown to be nonnegative in $B_\delta$ and $B_\delta$ is an open set, $z^\star$ is a local minimum. Thus, $\grad t_h(z^\star) = 0$ and $z^\star$ is a fixed point. This again contradicts local asymptotic stability.
\end{proof}

\section{Complexity of several qualitative properties of polynomial vector fields}\label{sec:poly}

As we remarked earlier, NP-hardness of testing local and global asymptotic stability of polynomial vector fields of degree 3 has already been established in our earlier work~\cite{AAA_ACC12_Cubic_Difficult},\cite[Chap. 4]{AAA_PhD}. In this section, we prove that deciding several other important properties of polynomial vector fields is also NP-hard. For many of these properties, our proof of NP-hardness builds on the proof in~\cite{AAA_ACC12_Cubic_Difficult}.  Whenever a property has to do with an equilibrium point, we take this equilibrium point to be at the origin. In what follows, the norm $||.||$ is always the Euclidean norm, and the notation $B_r$ denotes the ball of radius $r$; i.e., $B_r\mathrel{\mathop:}=\{x\ |\ ||x||\leq r\}$.

\begin{theorem}\label{thm:poly.hardness.results}
For polynomial differential equations of degree $d$ (with $d$ specified below), the following properties are NP-hard to decide:\\
(a) $d=3$, \emph{Invariance of a ball}: $\forall x(0)$ with $||x(0)||\leq 1$, $$||x(t)||\leq 1, \ \ \forall t\geq0.$$\\
(b) $d=1$, \emph{Invariance of a basic semialgebraic set defined by a quartic polynomial}: $\forall x(0)\in\mathcal{S}$, 
$$x(t)\in\mathcal{S}, \ \ \forall t\geq0,$$ where $\mathcal{S}\mathrel{\mathop:}=\{x\ |\ p(x)\leq 1\}$ and $p$ is a given form of degree four.\\
(c) $d=3$, \emph{Inclusion of a ball in the region of attraction of an equilibrium point}: $\forall x(0)$ with $||x(0)||\leq 1,$  $$x(t)\rightarrow 0, \ \mbox{as}\ \  t\rightarrow\infty.$$ \\
(d) $d=3$, \emph{Local attractivity of an equilibrium point}: $\exists \delta>0$ such that $\forall x(0)\in B_\delta,$ $$x(t)\rightarrow 0, \ \mbox{as}\ \  t\rightarrow\infty.$$ \\
(e) $d=4$, \emph{Stability of an equilibriym point in the sense of Lyapunov}: $\forall \epsilon>0$, $\exists\delta=~\delta(\epsilon)$ such that $$||x(0)||<\delta\ \Rightarrow ||x(t)||<\epsilon, \ \ \forall t\geq0.$$ \\
(f) $d=3$, \emph{Boundedness of trajectories}: $\forall x(0)$, $\exists r=r(x(0))$ such that
$$||x(t)||<r, \ \ \forall t\geq0.$$ \\
(g) $d=3$, \emph{Existence of a local quadratic Lyapunov function}: $\exists\delta>0$ and a quadratic Lyapunov function $V(x)=x^TPx$ such that $V(x)>0$ for all $x\in B_\delta, \ x\neq0$ (or equivalently $P\succ 0$), and $$\dot{V}(x)<0, \ \ \forall x\in B_\delta, \ x\neq0.$$\\
(h) $d=4$, \emph{Local collision avoidance}: $\exists\delta>0$ such that $\forall x(0)\in B_\delta$, $$x(t)\notin\mathcal{S}, \  \ \forall t\geq0,$$
where $\mathcal{S}$ is a given polytope.\\
(i) $d=3$, \emph{Existence of a stabilizing controller}: There exists a particular (e.g. smooth, or polynomial of fixed degree) control law $u(x)$ that makes the origin of $$\dot{x}=f(x)+g(x)u(x)$$ locally asymptotically stable, where $f$ and $g\neq 0$ here have degrees $3$.

\end{theorem}

\begin{proof} (a) The proofs of parts (a) and (b) of the theorem are based on a reduction from the polynomial nonnegativity problem: given a (homogeneous) polynomial $p:\mathbb{R}^n\rightarrow\mathbb{R}$, decide whether $p(x)\geq0, \forall x\in\mathbb{R}^n$? If the degree of $p$ is even and larger or equal than $4$, the problem is well-known to be NP-hard. This follows e.g. as an immediate consequence of NP-hardness testing matrix copositivity~\cite{nonnegativity_NP_hard}. (The original proof of NP-hardness of checking matrix copositivity in~\cite{nonnegativity_NP_hard} is via a reduction from the subset sum problem and only establishes weak NP-hardness. However, reductions from the stable set problem to matrix copositivity are also
known~\cite{deKlerk_StableSet_copositive} and they result in NP-hardness in the strong sense.)

We now proceed with the proof of (a). Given a quartic form $p$, we construct the vector field $$\dot{x}=-\nabla p(x).$$ Note that the vector field has degree $3$ and is homogeneous. We claim that the unit ball $B_1$ is invariant under the trajectories of this system if and only if $p$ is nonnegative. This of course establishes the desired NP-hardness result. To prove the claim, consider the function $V(x)\mathrel{\mathop:}=||x||^2.$ Clearly, $B_1$ is invariant under the trajectories of $\dot{x}=-\nabla p(x)$ if and only if $\dot{V}(x)\leq 0$ for all $x$ with $||x||=1.$ Since $\dot{V}$ is a homogeneous polynomial, this condition is equivalent to having $\dot{V}$ nonpositive for all $x\in\mathbb{R}^n$. However, from Euler's identity we have $$\dot{V}(x)=\langle 2x, -\nabla p(x) \rangle=-8p(x).$$

(b) Once again, we provide a reduction from the problem of checking nonnegativity of quartic forms. Given a quartic form $p$, we let the set $\mathcal{S}$ be defined as $\mathcal{S}=\{x\ |\ p(x)\leq 1\}$. Let us consider the linear dynamical system $$\dot{x}=-x.$$ We claim that $\mathcal{S}$ is invariant under the trajectories of this linear system if and only if $p$ is nonnegative. To see this, consider the derivative $\dot{p}$ of $p$ along the trajectories of $\dot{x}~=~-~x$ and note its homogeneity. With the same reasoning as in the proof of part (a), $\dot{p}(x)\leq 0$ for all $x\in\mathbb{R}^n$ if and only if the set $\mathcal{S}$ is invariant. From Euler's identity, we have
$$\dot{p}(x)=\langle \nabla p(x), -x \rangle=-4p(x).$$ Note that the role of the dynamics and the gradient of the ``Lyapunov function'' are swapped in the proofs of (a) and (b). 

{\bf A precursor to the proofs of (c)-(h).} The proofs of (c)-(h) are all via reductions from the problem of testing local asymptotic stability (las) of cubic vector fields. This problem has been shown to be NP-hard in~\cite{AAA_ACC12_Cubic_Difficult}. All of the reductions that follow are rather simple but subtly rely on the specific construction in~\cite{AAA_ACC12_Cubic_Difficult}. For this reason, we need to first recall some elements of that construction.

In a manner similar to that of the proof of Theorem~\ref{thm:las.trig}, the reduction in~\cite{AAA_ACC12_Cubic_Difficult} takes an instance of ONE-IN-THREE 3SAT and constructs a homogeneous polynomial $V:\mathbb{R}^n\rightarrow\mathbb{R}$ of degree $4$ that is positive definite if and only if the ONE-IN-THREE 3SAT instance is not satisfiable. Then, a cubic vector field is constructed as follows:
\begin{equation}\label{eq:xdot=f(x)=-gradV(x)}
\dot{x}\mathrel{\mathop:}=f(x)=-\nabla V(x).
\end{equation}
We recall the following facts from~\cite{AAA_ACC12_Cubic_Difficult} about this vector field:
\begin{enumerate}
\item The vector field in (\ref{eq:xdot=f(x)=-gradV(x)}) is homogeneous and therefore it is locally asymptotically stable if and only if it is globally asymptotically stable.
\item The vector field is (locally or globally) asymptotically stable if and only if $V$ is positive definite. Hence, the system, if asymptotically stable, by construction always admits a quartic Lyapunov function.
\item If the vector field is \emph{not} asymptotically stable, then there always exists a nonzero point $\bar{x}\in\{0,1\}^n$ such that $f(\bar{x})=0$; i.e., $\bar{x}$ is a nonzero equilibrium point.
\end{enumerate}
We now proceed with the proofs of (c)-(h). In what follows, $f(x)$ will always refer to the vector field in (\ref{eq:xdot=f(x)=-gradV(x)}).

(c) The claim is an obvious implication of the homogeneity of $f$. Since $f(\lambda x)=\lambda^3 f(x)$, for all $\lambda\in\mathbb{R}$ and all $x\in\mathbb{R}^n$, the origin is las if and only if for any $r$, all trajectories in $B_r$ converge to the origin.\footnote{For a general cubic vector field, validity of property (c) for a particular value of $r$ is of course not necessary for local asymptotic stability. The reader should keep in mind that the class of homogeneous cubic vector fields is a subset of the class of all cubic vector fields, and hence any hardness result for this class immediately implies the same hardness result for all cubic vector fields.}

(d) If $f$ is las, then of course it is by definition locally attractive. On the other hand, if $f$ is not las, then $f(\bar{x})=0$ for some nonzero $\bar{x}\in\{0,1\}^n$. By homogeneity of $f$, this implies that $f(\alpha\bar{x})=0, \ \forall\alpha\geq 0.$ Therefore, arbitrarily close to the origin we have stationary points and hence the origin cannot be locally attractive.

(e) Let $x^4=(x_1^4,\ldots, x_n^4)^T.$ Consider the vector field
\begin{equation}\label{eq:xdot=f(x)+x^4}
\dot{x}=f(x)+x^4.
\end{equation}
We claim that the origin of (\ref{eq:xdot=f(x)+x^4}) is stable in the sense of Lyapunov if and only if the origin of (\ref{eq:xdot=f(x)=-gradV(x)}) is las. Suppose first that (\ref{eq:xdot=f(x)=-gradV(x)}) is not las. Then we must have $f(\alpha\bar{x})=0$ for some nonzero $\bar{x}\in\{0,1\}^n$ and $\forall\alpha\geq 0$. Therefore for the system (\ref{eq:xdot=f(x)+x^4}), trajectories starting from \emph{any} nonzero point on the line connecting the origin to $\bar{x}$ shoot out to infinity while staying on the line. (This is because on this line, the dynamics are simply $\dot{x}=x^4.$) As a result, stability in the sense of Lyapunov does not hold as there exists an $\epsilon>0$ (in fact for any $\epsilon>0$), for which $\nexists\delta>0$ such that trajectories starting in $B_\delta$ stay in $B_\epsilon$. Indeed, as we argued, arbitrarily close to the origin we have points that shoot out to infinity. 

Let us now show the converse. If (\ref{eq:xdot=f(x)=-gradV(x)}) is las, then $V$ is indeed a strict Lyapunov function for it; i.e. it is positive definite and has a negative definite derivative $-||\nabla V(x)||^2$. Using the same Lyapunov function for the system in (\ref{eq:xdot=f(x)+x^4}), we have $$\dot{V}(x)=-||\nabla V(x)||^2+\langle \nabla V(x), x^4\rangle.$$ Note that the first term in this expression is a homogeneous polynomial of degree $6$ while the second term is a homogeneous polynomial of degree $7$. Negative definiteness of the lower order term implies that there exists a positive real number $\delta$ such that $\dot{V}(x)<0$ for all nonzero $x\in B_\delta$. This together with positive definiteness of $V$ implies via Lyapunov's theorem that (\ref{eq:xdot=f(x)+x^4}) is las and hence stable in the sense of Lypunov.

(f) Consider the vector field
\begin{equation}\label{eq:xdot=f(x)+x}
\dot{x}=f(x)+x.
\end{equation}
We claim that the trajectories of (\ref{eq:xdot=f(x)+x}) are bounded if and only if the origin of (\ref{eq:xdot=f(x)=-gradV(x)}) is las. Suppose (\ref{eq:xdot=f(x)=-gradV(x)}) is not las. Then, as in the previous proof, there exists a line connecting the origin to a point $\bar{x}\in\{0,1\}^n$ such that trajectories on this line escape to infinity. (In this case, the dynamics on this line is governed by $\dot{x}=x$.) Hence, not all trajectories can be bounded. For the converse, suppose that (\ref{eq:xdot=f(x)=-gradV(x)}) is las. Then $V$ (resp. $-||\nabla V(x)||^2$) must be positive (resp. negative) definite. Now if we consider system (\ref{eq:xdot=f(x)+x}), the derivative of $V$ along its trajectories is given by $$\dot{V}(x)=-||\nabla V(x)||^2+\langle \nabla V(x), x\rangle.$$ Since the first term in this expression has degree $6$ and the second term degree $4$, there exists an $r$ such that $\dot{V}<0$ for all $x\notin B_r$. This condition however implies boundedness of trajectories; see e.g.~\cite{Higher_Derive_Lagrange_Stability}.

(g) If $f$ is not las, then there cannot be any local Lyapunov functions, in particular not a quadratic one. If $f$ is las, then we claim the quadratic function $W(x)=||x||^2$ is a valid (and in fact global) Lyapunov function for it. This can be seen from Euler's identity $$\dot{W}(x)=\langle 2x, -\nabla V(x)  \rangle=-8V(x),$$
and by noting that $V$ must be positive definite.

(h) We define our dynamics to be the one in (\ref{eq:xdot=f(x)+x^4}), and the polytope $\mathcal{S}$ to be $$\mathcal{S}=\{x\ |\ x_i\geq 0, 1\leq \sum_{i=1}^n x_i \leq 2\}.$$ 

Suppose first that $f$ is not las. Then by the argument given in (e), the system in (\ref{eq:xdot=f(x)+x^4}) has trajectories that start out arbitrarily close to the origin (at points of the type $\alpha\bar{x}$ for some $\bar{x}\in\{0,1\}^n$ and for arbitrarily small $\alpha$), which exit on a straight line to infinity. Note that by doing so, such trajectories must cross $\mathcal{S}$; i.e. there exists a positive real number $\bar{\alpha}$ such that $1\leq \sum_{i=1}^n \bar{\alpha}\bar{x}_i \leq 2.$ Hence, there is no neighborhood around the origin whose trajectories avoid $\mathcal{S}$. 

For the converse, suppose $f$ is las. Then, we have shown while proving (e) that (\ref{eq:xdot=f(x)+x^4}) must also be las and hence stable in the sense of Lyapunov. Therefore, there exists $\delta>0$ such that trajectories starting from $B_\delta$ do not leave $B_{\frac{1}{2}}$---a ball that is disjoint from $\mathcal{S}$.

(i) Let $f$ be as in (\ref{eq:xdot=f(x)=-gradV(x)}) and $g(x)=(x_1x_2^2-x_1^2x_2)\boldsymbol{1}\boldsymbol{1} ^T$, where $\boldsymbol{1}$ denotes the vector of all ones. The following simple argument establishes the desired NP-hardness result irrespective of the type of control law we may seek (e.g. linear control law, cubic control law, smooth control law, or anything else). If $f$ is las, then of course there exists a stabilizing controller, namely $u=0$. If $f$ is not las, then it must have an equilibrium point at a nonzero point $\bar{x}\in\{0,1\}^n$. Note that by construction, $g$ vanishes at all such points. Since $g$ is homogeneous, it also vanishes at all points $\alpha\bar{x}$ for any scalar $\alpha$. Therefore, arbitrarily close to the origin, there are equilibrium points that the control law $u(x)$ cannot possibly remove. Hence there is no controller that can make the origin las.
\end{proof}

\begin{remark}
Arguments similar to the one presented in the proof of (i) above can be given to show NP-hardness of deciding existence of a controller that establishes several other properties, e.g., invariance of the unit ball, inclusion of the unit ball in the region of attraction, etc. In the statement of (h), the fact that the set $\mathcal{S}$ is a polytope is clearly arbitrary. This choice is only made because ``obstacles'' are most commonly modeled in the literature as polytopes. We also note that a related problem of interest here is that of deciding, given two polytopes, whether all trajectories starting in one avoid the other. This question is the complement of the usual reachability question, for which claims of undecidability have already appeared~\cite{Undecidability_vec_fields_survey}; see also~\cite{Skolem_Pisot_problem_CT}.
\end{remark}

\section{Conclusions and Open problems}
Under the assumption P$\neq$NP, we have shown the impossibility of polynomial time (or even pseudo-polynomial time) algorithms for ten decision problems that ubiquitously arise in control theory and the study of continuous time dynamical systems. Although our hardness results are valid even for very restricted classes of systems (e.g. gradient systems), it is of course still possible that these decision problems admit polynomial time algorithms for other special (and possibly important) \emph{subclasses} of polynomial or trigonometric vector fields.

Aside from extending the results of Theorem~\ref{thm:poly.hardness.results} to other classes of vector fields (such as trigonometric ones), the obvious class of questions that our work leaves open is to investigate the \emph{decidability} of the decision problems studied in Theorem~\ref{thm:poly.hardness.results}, or their NP-hardness for polynomial vector fields of degree one or two. Although for linear systems some of these questions become easy, we expect that our hardness results can be strengthened to the case when the degree is $2$. Quadratic vector fields already demonstrate very complex behaviour; for example, their stability does not imply existence of a polynomial Lyapunov function of any degree~\cite{AAA_MK_PP_CDC11_no_Poly_Lyap}. In general, one can reduce the degree of any vector field to two by introducing polynomially many new variables (see~\cite{Undecidability_vec_fields_survey}). However, this operation may or may not preserve the property of the vector field which is of interest.

\bibliographystyle{abbrv}
\bibliography{pablo_amirali}

\def\cprime{$'$}
\begin{thebibliography}{10}

\bibitem{AAA_PhD}
A.~A. Ahmadi.
\newblock {\em Algebraic relaxations and hardness results in polynomial
  optimization and {L}yapunov analysis}.
\newblock PhD thesis, Massachusetts Institute of Technology, September 2011.

\bibitem{AAA_ACC12_Cubic_Difficult}
A.~A. Ahmadi.
\newblock On the difficulty of deciding asymptotic stability of cubic
  homogeneous vector fields.
\newblock In {\em Proceedings of the 2012 American Control Conference}, 2012.

\bibitem{AAA_MK_PP_CDC11_no_Poly_Lyap}
A.~A. Ahmadi, M.~Krstic, and P.~A. Parrilo.
\newblock A globally asymptotically stable polynomial vector field with no
  polynomial {L}yapunov function.
\newblock In {\em Proceedings of the 50$^{th}$ IEEE Conference on Decision and
  Control}, 2011.

\bibitem{Arnold_Problems_for_Math}
V.~I. Arnold.
\newblock Problems of present day mathematics{, XVII (Dynamical systems and
  differential equations)}.
\newblock {\em Proc. Symp. Pure Math.}, 28(59), 1976.

\bibitem{Atkinson09}
K.~Atkinson, W.~Han, and D.~Stewart.
\newblock {\em Numerical Solution of Ordinary Differential Equations},
  volume~81.
\newblock John Wiley \& Sons Inc, 2009.

\bibitem{Skolem_Pisot_problem_CT}
P.~Bell, J.-C. Delvenne, R.~Jungers, and V.~D. Blondel.
\newblock The continuous {S}kolem-{P}isot problem: on the complexity of
  reachability for linear ordinary differential equations.
\newblock {\em Theoretical Computer Science}, 411:3625--3634, 2010.

\bibitem{Bequette98}
B.~Bequette.
\newblock {\em Process Dynamics: Modeling, Analysis, and Simulation}.
\newblock Prentice Hall PTR, 1998.

\bibitem{BlTi_complexity_3classes}
V.~D. Blondel and J.~N. Tsitsiklis.
\newblock Overview of complexity and decidability results for three classes of
  elementary nonlinear systems.
\newblock In {\em Learning, Control and Hybrid Systems}, pages 46--58.
  Springer, 1998.

\bibitem{BlTi1}
V.~D. Blondel and J.~N. Tsitsiklis.
\newblock A survey of computational complexity results in systems and control.
\newblock {\em Automatica}, 36(9):1249--1274, 2000.

\bibitem{Survey_CT_Computation}
O.~Bournez and M.~L. Campagnolo.
\newblock A survey on continuous time computations.
\newblock {\em New Computational Paradigms}, 4:383--423, 2008.

\bibitem{Braun93}
M.~Braun.
\newblock {\em Differential Equations and their Applications: An Introduction
  to Applied Mathematics}, volume~11.
\newblock Springer, 1993.

\bibitem{Costa_Doria_undecidabiliy}
N.~C.~A. {da Costa} and F.~A. Doria.
\newblock On {Arnold's Hilbert} symposium problems.
\newblock In {\em Computational Logic and Proof Theory}, volume 713 of {\em
  Lecture Notes in Computer Science}, pages 152--158. Springer, 1993.

\bibitem{deKlerk_StableSet_copositive}
E.~de~Klerk and D.~V. Pasechnik.
\newblock Approximation of the stability number of a graph via copositive
  programming.
\newblock {\em SIAM Journal on Optimization}, 12(4):875--892, 2002.

\bibitem{Freedman80}
H.~Freedman.
\newblock {\em Deterministic Mathematical Models in Population Ecology}.
\newblock M. Dekker, 1980.

\bibitem{GareyJohnson_Book}
M.~R. Garey and D.~S. Johnson.
\newblock {\em Computers and Intractability}.
\newblock W. H. Freeman and Co., San Francisco, Calif., 1979.

\bibitem{Undecidability_vec_fields_survey}
E.~Hainry.
\newblock Decidability and undecidability in dynamical systems.
\newblock Research report. Available at
  \texttt{\texttt{http://hal.inria.fr/inria-00429965/PDF/dynsys.pdf}}, 2009.

\bibitem{Higher_Derive_Lagrange_Stability}
J.~A. Heinen and M.~Vidyasagar.
\newblock Lagrange stability and higher order derivatives of {L}iapunov
  functions.
\newblock {\em IEEE Trans. Automatic Control}, 58(7):1174, 1970.

\bibitem{Khalil:3rd.Ed}
H.~Khalil.
\newblock {\em Nonlinear systems}.
\newblock Prentice Hall, 2002.
\newblock Third edition.

\bibitem{Murray94}
R.~M. Murray, Z.~Li, and S.~S. Sastry.
\newblock {\em A Mathematical Introduction to Robotic Manipulation}.
\newblock CRC Press, Inc., 1994.

\bibitem{nonnegativity_NP_hard}
K.~G. Murty and S.~N. Kabadi.
\newblock Some {NP}-complete problems in quadratic and nonlinear programming.
\newblock {\em Mathematical Programming}, 39:117--129, 1987.

\bibitem{Sontag_complexity_comparison}
E.~D. Sontag.
\newblock From linear to nonlinear: some complexity comparisons.
\newblock In {\em Proceedings of the 34$^{th}$ IEEE Conference on Decision and
  Control}, 1995.

\end{thebibliography}

\end{document}